\documentclass[11pt]{article}
\usepackage{latexsym, amssymb, enumerate, amsmath, amsthm,multicol}
\usepackage[utf8]{inputenc}
\usepackage[T1]{fontenc}
\usepackage{fullpage}
\usepackage{layout}

\usepackage{cleveref}
\usepackage{cite}
\usepackage{url}
\usepackage{esint}
\usepackage{dsfont}
\usepackage{color}
\usepackage{overpic}
\usepackage{soul}

\usepackage{graphicx,epsfig}

\crefname{assumption}{assumption}{assumptions}
\crefname{thm}{theorem}{theorems}
\crefname{lem}{lemma}{lemmas}
\crefname{cor}{corollary}{corollaries}
\crefname{prop}{proposition}{propositions}
\Crefname{theorem}{Theorem}{Theorems}
\crefname{conjecture}{conjecture}{conjectures}

\newcommand{\G}{\Gamma}
\newcommand{\tM}{\tilde{M}}
\newcommand{\na}{\nabla}

\newcommand{\R}{\mathbb{R}}

\newcommand{\vp}{\varphi}

\newcommand{\Sing}{\mathrm{Sing}\left(M,\Gamma\right)}
\newcommand{\di}{\mathrm{div}}
\renewcommand{\epsilon}{\varepsilon}
\newcommand{\eps}{\varepsilon}

\newtheorem{thm}{Theorem}[section]

\newtheorem{prop}[thm]{Proposition}

\newtheorem{lemma}[thm]{Lemma}

\newtheorem{definition}[thm]{Definition}
\newtheorem{proposition}[thm]{Proposition}

\newtheorem{example}[thm]{Example}

\newtheorem{rem}[thm]{Remark}
\newtheorem{theorem}[thm]{Theorem}

\date\today
\author{Nils Caillerie
\footnote{Department of mathematics and statistics, Georgetown University, Saint Mary's Hall, 3700 O Street NW, Washington, DC 20057, United States of America. E-mail: \texttt{nc691@georgetown.edu}}}

\begin{document}
\title{Large deviations of a forced velocity-jump process with a Hamilton-Jacobi approach}
\maketitle

\begin{abstract}
We study the dispersion of a particle whose motion dynamics can be described by a forced velocity jump process. To investigate large deviations results, we study the Chapman-Kolmogorov equation of this process in the hyperbolic scaling $(t,x,v)\to(t/\eps,x/\eps,v)$ and then, perform a Hopf-Cole transform which gives us a kinetic equation on a potential. We prove the convergence of this potential to the solution of a Hamilton-Jacobi equation. The hamiltonian can have a $C^1$ singularity, as was previously observed in this kind of studies. This is a preliminary work before studying spreading results for more realistic processes.
\end{abstract}
\noindent{\bf AMS Class. No:} {35Q92, 45K05, 35D40, 35F21}

\section{Introduction}

In this paper, we study the dispersion in $\R^d$ of a particle whose motion dynamics is described by the following piecewise deterministic Markov process (PDMP). During the so-called "run phase" (\em i.e \em the deterministic part), the particle is moving in $\R^d$ and is submitted to a force whose intensity and direction are given by the vector $\G$, which only depends on the insantaneous velocity of the particle. Therefore, its position $\mathcal{X}_s$ and velocity $\mathcal{V}_s$ at time $s$ are given by the following system of ODEs
\begin{equation*}
\begin{cases}
\overset{\cdot}{\mathcal{X}}_s=\mathcal{V}_s,\\
\overset{\cdot}{\mathcal{V}}_s=\Gamma(\mathcal{V}_s).
\end{cases}
\end{equation*}

We shall call the measure space $(V,\nu)$ the set of admissible velocities. After a random exponential time with mean 1, a "tumble" occurs: the particle chooses a new velocity at random on the space $V$, independently from its last velocity. The law of the velocity redistribution process is given by the probability density function $M$. The particle enters a new running phase which will again last for a random exponential time of parameter 1, and so on.

The Kolmogorov equation of this process is the following conservative kinetic equation:
\begin{equation}\label{eq1}
\partial_t f +v\cdot \na_x f + \di_v\left( \Gamma f \right)= M(v)\rho - f,\quad (t,x,v)\in \R_+\times\R^d\times V,
\end{equation}
where $\rho$ is the macroscopic density of $f$:
\begin{equation*}
\rho(t,x):=\int_V f(t,x,v)d\nu(v).
\end{equation*}
We assume that $\G d\overrightarrow{S}$ is the null measure, where the vector $\overrightarrow{S}(v)$ is the normal vector to $\partial V$ at point $v\in \partial V$. This condition guarantees that the total mass of the system is conserved thanks to Ostrogradsky's Theorem. We also assume that $V$ is a compact manifold of $\R^d$ with a (possibly empty) boundary. In the case where the boundary of $V$ is not empty, for every function $g$ from $V$ to $\mathbb{R}$, we define (when possible) $\Gamma\cdot \nabla_v g$ on $\partial V$ as follows:
\begin{equation*}
(\Gamma\cdot \nabla_v g)(w)=\left. \frac{d}{ds} g(\gamma_s)\right\vert_{s=0},
\end{equation*}
for $\gamma$ in $V$ such that $\gamma(0)=w$ and $\overset{\cdot}{\gamma_s}=\Gamma(\gamma_s)$ for all $s$ in $[-\delta,\delta]$. For a function $G$ from $V$ to $\mathbb{R}^d$, we define $\di_v (\Gamma G)$ on $\partial V$ (when possible) as
\begin{equation*}
\di_v (\Gamma G)(w)= \sum_i(\Gamma \nabla_v G_i)(w).
\end{equation*}

 The function $M\in C^0(V)$ is assumed to satisfy
\begin{equation}\label{infM}
\underset{v\in V}{\mathrm{min}}\,M(v)>0.
\end{equation}

The so-called \em force term \em $\Gamma$ is a lipschitz-continuous function of $v$. We assume that there exists $\alpha>0$ such that
\begin{equation}\label{divgamma}
1+\di\Gamma\left(v\right)\geq\alpha>0,\quad\forall v\in V.
\end{equation}
We introduce the flow of $-\Gamma$:
\begin{equation}\label{flowgamma}
\begin{cases}
\overset{\cdot}{\phi^v_s}=-\Gamma\left(\phi^v_s\right),\\
\phi^v_0=v,
\end{cases}
\end{equation}
We assume that it satisfies a Poincar\'e-Bendixson condition in the sense that, for all $v\in V$, the limit set of orbit of $v$ is either a zero of $-\Gamma$ or a periodic orbit of $-\Gamma$. In other words,
\begin{equation}\label{Poincare}
\forall v \in V,\quad \exists w_0 \in V\mbox{ and }T\geq0\mbox{ such that }\phi^{w_0}_{T}=w_0\mbox{ and } \bigcap_{t>0}\overline{\left\{ \phi^v_s,\: s\geq t \right\}}=\left\{ \phi^{w_0}_s,\: 0\leq s \leq T \right\}.
\end{equation}
Finally, we assume the following mixing property:
\begin{equation}\label{mixing}
\forall v \in V,\;\exists w(v)\in V\;\mbox{ such that }\;\forall F\in C^0(V,\R),\quad\underset{t\to+\infty}{\mathrm{lim}}\;\frac{1}{t}\int_0^t F(\phi^v_s)ds=F(w).
\end{equation}
Note that, thanks to the Poincar\'e -Bendixson condition (\ref{Poincare}), we already get the existence of a $w(v)$ in the convex hull of $V$ such that $\frac{1}{t}\int_{[0,t]}F(\phi^v_s)ds\to F(w)$. Here, we assume furthermore that this "representative" of $v$ can be chosen in $V$, even when $V$ is not convex.

In order to study large deviations results for this process, we use the method of geometric optics \cite{evans_pde_1987,freidlin_geometric_1986}. We study the rescaled function $f^\eps(t,x,v):=f\left(\frac{t}{\eps},\frac{x}{\eps},v\right)$, which satisfies
\begin{equation*}
\partial_t f^{\eps}+v\cdot \na f^{\eps} +\frac{1}{\eps}\di_v\left(\G f^{\eps}\right)=\frac{1}{\eps}\left(M(v)\rho^{\eps}-f^{\eps}\right),\quad(t,x,v)\in\R_+\times\R^d\times V.
\end{equation*}
The function $f^{\eps}$ quickly relaxes towards $\tM$, the solution of
\begin{equation}\label{tm}
\begin{cases}
\di_v\left(\G(v) \tM(v)\right)= M(v)\int_V \tM(v')d\nu(v')-\tM(v),\\
\int_V \tM(v')d\nu(v')=1.
\end{cases}
\end{equation}
We introduce the following WKB ansatz
\begin{equation*}
\varphi^{\eps}(t,x,v):=-\eps \mathrm{log}\left(\frac{f^{\eps}(t,x,v)}{\tM}\right),\quad\mbox{or equivalently,}\quad f^{\eps}(t,x,v)=\tM e^{-\frac{\varphi^{\eps}(t,x,v)}{\eps}}.
\end{equation*}
Then, $\varphi^{\eps}$ satisfies
\begin{equation}\label{main}
\partial_t \varphi^{\eps}+v\cdot \na \vp^{\eps} +\frac{\G}{\eps}\cdot \na_v \varphi^{\eps}=\frac{M}{\tM}\int_V\tM(v')\left(1-e^{\frac{\vp^{\eps}-\vp'^{\eps}}{\eps}}\right)d\nu(v'),
\end{equation}
where (here and until the end) $\vp'^{\eps}$ stands for $\vp^{\eps}(t,x,v')$.

The main result of this paper is that $(\varphi^{\eps})_\eps$ converges to the viscosity solution of some Hamilton-Jacobi equation.

\subsection*{Motivations and earlier related works}

The motivation of this work comes from the study of concentration waves in bacterial colonies of \em Escherichia coli\em. Kinetic models have been proposed to study the Run \& Tumble motion of the bacterium at the mesoscopic scale in \cite{alt_biased_1980,stroock_stochastic_1974}. More recently, it has been established that these kinetic models are more accurate than their diffusion approximations to describe the speed of a colony of bacteria in a channel of nutrient \cite{saragosti_directional_2011}. This has raised some interest on the study of front propagation in kinetic models driven by chemotactic effect \cite{calvez_chemotactic_2016} but also by growth effect \cite{bouin_hamilton-jacobi_2015,bouin_propagation_2015,bouin_spreading_2017}. Our goal is to explore those studies further, by considering kinetic equations with a force term, in view of studying propagation of biological species with an effect of the environment (one could think of fluid resistance of water for bacteria, for example). A physically relevant force term may not satisfy all the assumptions of the present paper, mostly because of (\ref{divgamma}), but our result and methods can be adapted for different force terms. Therefore, our study should be considered as a preliminary work before studying more realistic models.

When $\G\equiv 0$, a convergence result for $(\varphi^\eps)_\eps$ already exists. The question has originally been solved in \cite{bouin_kinetic_2012} by Bouin and Calvez who proved convergence of $(\vp^\eps)_\eps$ to the solution of a Hamilton-Jacobi equation with an implicitly defined hamiltonian. Their result, however, only holds in dimension 1, since the implicit formulation of the hamiltonian may not have a solution. It was then generalized to higher dimensions by the author in \cite{caillerie_large_2017-1}. The proof relied on the establishment of uniform (with respect to $\eps$) \em a priori \em bounds on the potential $\varphi^\eps$, which may not hold in our situation. If one requires that $\mathrm{div}\Gamma=0$, the proof of \cite{caillerie_large_2017-1} can be adapted to our situation since one can establish those \em a priori \em bounds (see \cite{caillerie_stochastic_2017}, Chapter 3).

When the velocity set is unbounded and $\Gamma\equiv 0$, one observes an acceleration of the front of propagation, which highlights the difference between the kinetic model and its diffusion approximation. Due to this acceleration, the hyperbolic scaling is no longer the right one to follow the front. In the special case where $M$ is gaussian and for the scaling $(t,x,v)\to(\frac{t}{\eps},\frac{x}{\eps^{3/2}},\frac{v}{\eps^{1/2}})$ the Hamilton-Jacobi limit was performed by Bouin, Calvez, Grenier and Nadin in \cite{bouin_large_2016}.

As was previously mentioned, spreading can also been driven by growth effect. Propagation in a similar model, without the force term but with a reaction-term of KPP-type was investigated by Bouin, Calvez and Nadin in \cite{bouin_propagation_2015}. They established the existence of travelling wave solutions in the one-dimensional case. Interestingly enough, the speed of propagation they established differed from the KPP speed obtained in the diffusion approximation. Their result was generalized to the higher velocity dimension case by Bouin and the author in \cite{bouin_spreading_2017}. In the present paper, we will use the method of geometric optics \cite{evans_pde_1987,freidlin_geometric_1986}, the half-relaxed limits method of Barles and Perthame \cite{barles_exit_1988} and the perturbed test function method of Evans \cite{evans_perturbed_1989} in a similar fashion as in \cite{bouin_spreading_2017,bouin_hamiltonjacobi_2015}. The Hamilton-Jacobi framework can also be used in other various situations involving population dynamics (not necessarily structured by velocity) \cite{bouin_invasion_2012,bouin_hamiltonjacobi_2015,mirrahimi_time_2015,mirrahimi_asymptotic_2015,gandon_hamiltonjacobi_2017}.
%
%

\subsection*{Main result}

To identify a candidate for the limit, let us assume formally that this limit $\varphi^0:=\underset{\eps\to0}{\mathrm{lim}}\,\varphi^{\eps}$ is independent of the velocity variable and that the convergence speed is of order 1 in $\eps$, which would mean that there exists a function $\eta$ such that
\begin{equation}\label{hilbert}
\varphi^{\eps}(t,x,v)=\varphi^0(t,x)+\eps \eta(t,x,v)+\mathcal{O}(\eps^2).
\end{equation}
Plugging (\ref{hilbert}) into \Cref{main}, we get formally at the order $\eps^0$
\begin{equation}\label{prespec}
\partial_t \varphi^0 +v\cdot \na_x \varphi^0+\G\cdot\na_v \eta = \frac{M}{\tM}\int_V \tM'\left(1-e^{\eta-\eta'}\right)d\nu(v').
\end{equation}
When $t$ and $x$ are fixed, (\ref{prespec}) is a differential equation in the variable $v$. Let us set $p:=\na_x\varphi^0(t,x)$, $H:=-\partial_t \varphi^0(t,x)$ and $Q(v):=e^{-\eta(v)}$. Then, $Q$ satisfies
\begin{equation}\label{probspec}
\begin{cases}
HQ(v)=\left(v\cdot p - \frac{M(v)}{\tM(v)}\right)Q(v)-\G(v)\cdot\na_v Q(v)+\frac{M(v)}{\tM(v)}\int_V \tM'Q'd\nu(v'),&v\in V,\\
Q>0.
\end{cases}
\end{equation}

For fixed $p$, this is a spectral problem where $Q$ and $H$ are viewed as an eigenvector and the associated eigenvalue. We will discuss the resolution of this spectral problem in \Cref{identification}. This resolution motivates the introduction of the following hamiltonian.
\begin{definition}\label{defiprincipale}
For all $p\in\R^d$, we set
\begin{equation}\label{hamiltonian}
\mathcal{H}(p):=\mathrm{inf}\left\{ H\in\R,\quad \int_V\tM(v')Q_{p,H}(v')d\nu(v') \leq 1 \right\},
\end{equation}
where
\begin{equation}\label{defQ}
Q_{p,H}\left(v\right):=\int_0^{+\infty}\frac{M(\phi^v_t)}{\tM(\phi^v_t)}\mathrm{exp}\left(-\int_0^t\left(\frac{M(\phi^v_s)}{\tM(\phi^v_s)}+H-\phi^v_s\cdot p\right)ds\right)dt,
\end{equation}
and $\phi$ is the flow of $-\Gamma$:
\begin{equation}\label{flow}
\begin{cases}
\overset{\cdot}{\phi^v_s}=-\Gamma\left(\phi^v_s\right),\\
\phi^v_0=v.
\end{cases}
\end{equation}
\end{definition}

As in \cite{coville_singular_2013-1}, this spectral problem may not have a solution in $C^1(V)$ and one may need to solve it in the set of positive measures (one can refer to \cite{caillerie_large_2017-1} where a similar situation occurs). Therefore, let us define the so-called singular set of $M$ and $\Gamma$, that is the set where the spectral problem has no solution in $C^1(V)$:
\begin{definition}
We call "Singular set of $M$ and $\Gamma$" the set
\begin{equation}\label{sing}
\Sing:=\left\{p\in\R^d,\; \left\{H\in \R,\;1<\int_V\tM'Q'_{p,H}d\nu(v')<+\infty \right\}=\emptyset \right\}.
\end{equation}
\end{definition}

Let us now state our main result.
\begin{theorem}\label{thm}
Let us assume that (\ref{infM}), (\ref{divgamma}), (\ref{Poincare}) and (\ref{mixing}) hold and that $\tM$ satisfies (\ref{tm}). Let $\varphi^{\eps}$ satisfy \Cref{main}. Let us assume furthermore that the initial condition is well-prepared: $\varphi^{\eps}(0,x,v):=\varphi_0(x)\geq 0$. Then, the function $\varphi^{\eps}$ converges uniformly locally toward some function $\varphi^0$ which is independent from $v$. Moreover, $\varphi^0$ is the viscosity solution of the Hamiton-Jacobi equation
\begin{equation}\label{HJ}
\begin{cases}
\partial_t \varphi^0+\mathcal{H}\left(\na_x \varphi^0\right)=0,& (t,x)\in \mathbb{R}_+^*\times \mathbb{R}^d,\\
\varphi^0(0,\cdot)=\varphi_0.
\end{cases}
\end{equation}
where $\mathcal{H}$ is defined as in \Cref{defiprincipale}.
\end{theorem}

\begin{rem}
The sufficient conditions on $\mathcal{H}$ that guarantee the uniqueness of the solution of the Hamilton-Jacobi Equation (\ref{HJ}) will be proven in \Cref{lipschitz}.
\end{rem}

The rest of the paper is organized as follows: in Section 2, we describe how we obtain the hamiltonian and prove some results that we will use later on. In particular, we solve the spectral problem \eqref{probspec}. Section 3 is devoted to the proof of \Cref{thm}.

\newpage

\section{Identification of the hamiltonian}\label{identification}

\subsection{Positivity and boundedness of $\tM$}

As a first step, we will prove that $0<\tM<+\infty$.

\begin{lemma}\label{borneinf}
Let us assume that (\ref{infM}) and (\ref{divgamma}) hold and let $\tM$ satisfy (\ref{tm}). Then, $\underset{v\in V}{\mathrm{min}}\,\tM(v)>0$.
\end{lemma}

\begin{proof}
Let $v_{\mathrm{min}}$ bet the point where $\tM$ reaches its minimum. Suppose $v_{\mathrm{min}}$ is in the interior of $V$. Then, $\na_v \tM(v_{\mathrm{min}})=0$, thus
\begin{equation*}
\tM(v_{\mathrm{min}})\di \Gamma (v_{\mathrm{min}})=M(v_{\mathrm{min}})-\tM(v_{\mathrm{min}}),
\end{equation*}
which implies
\begin{equation*}
\tM(v_{\mathrm{min}})=\frac{M(v_{\mathrm{min}})}{1+\di \Gamma(v_{\mathrm{min}})}\geq \frac{\mathrm{min}\,M}{1+\di \Gamma(v_{\mathrm{min}})}>0.
\end{equation*}
If $v_{\mathrm{min}}\in \partial V$, we can conclude likewise. Indeed, if $\Gamma(v_{\mathrm{min}})=0$, then the result is trivial. If $v_{\mathrm{min}}\in \partial V$ and $\Gamma(v_{\mathrm{min}})\neq 0$, since $\Gamma(v)\cdot d\overrightarrow{S}(v)=0$ for all $v\in \partial V$, there exists $v_0\in V$, $v_1 \in V$ and $\delta>0$ such that
\begin{equation*}
\begin{cases}
\phi^{v_{\mathrm{min}}}_s\in V,& \forall s\in [-\delta,\delta],\\
\phi^{v_{\mathrm{min}}}_{-\delta} = v_0,\\
\phi^{v_{\mathrm{min}}}_{\delta}=v_1.
\end{cases}
\end{equation*}
The extremal property of $(v_{\mathrm{min}})$ now implies that 
\begin{equation*}
\Gamma(v_{\mathrm{min}})\cdot \nabla_v \tM(v_{\mathrm{min}})=-\left. \frac{d}{ds}\tM(\phi^{v_{\mathrm{min}}}_s)\right\vert_{s=0}=0.
\end{equation*}

\end{proof}

\begin{lemma}\label{Lemmma}
Let us assume that (\ref{divgamma}) holds. Then,
\begin{equation*}
\underset{v\in V}{\mathrm{sup}}\,\tM(v)<+\infty.
\end{equation*}
\end{lemma}

\begin{proof}
We use the same ideas as in \Cref{borneinf} to prove that
\begin{equation*}
\tM(v_{\mathrm{max}})=\frac{M(v_{\mathrm{max}})}{1+\di\G(v_{\mathrm{max}})}\leq \frac{\left\Vert M \right\Vert_{\infty}}{\alpha},
\end{equation*}
thanks to (\ref{divgamma}).
\end{proof}

\subsection{The spectral problem}\label{subpart1}

Here, we discuss the resolution of the spectral problem, that is: for all $p\in\R^d$, find $H$ and a function $Q>0$ such that
\begin{equation*}
HQ(v)=\left(v\cdot p - \frac{M(v)}{\tM(v)}\right)Q(v)-\G(v)\cdot\na_v Q(v)+\frac{M(v)}{\tM(v)}\int_V \tM'Q'd\nu(v')
\end{equation*}
holds, for all $v\in V$.

To find such a solution, we use the method of characteristics: let us define $\phi$ as the flow of $-\G$ (see (\ref{flow})). Then, we have\\
$\displaystyle{\frac{d}{ds}\left(Q(\phi^v_s)\mathrm{exp}\left( -\int_0^s \left( \frac{M(\phi^v_{\sigma})}{\tM(\phi^v_\sigma)}+H-\phi^v_\sigma \cdot p \right) d\sigma\right)\right)}$
\begin{equation*}
=-\mathrm{exp}\left( -\int_0^s \left( \frac{M(\phi^v_\sigma)}{\tM(\phi^v_\sigma)}+H-\phi^v_\sigma \cdot p \right) d\sigma\right)\left[\left( \frac{M(\phi^v_s)}{\tM(\phi^v_s)}+H-\phi^v_s \cdot p \right)Q(\phi^v_s)+\Gamma(\phi^v_s)\cdot \na_v Q(\phi^v_s)\right]
\end{equation*}
\begin{flushright}
$\displaystyle{=-\mathrm{exp}\left( -\int_0^s \left( \frac{M(\phi^v_\sigma)}{\tM(\phi^v_\sigma)}+H-\phi^v_\sigma \cdot p \right) d\sigma\right)}\frac{M(\phi^v_s)}{\tM(\phi^v_s)}\int_V \tM'Q'd\nu(v').$
\par\end{flushright}
Suppose that 
\begin{equation*}
\displaystyle{\underset{s\to +\infty}{\mathrm{lim}}\mathrm{exp}\left( -\int_0^s \left( \frac{M(\phi^v_\sigma)}{\tM(\phi^v_\sigma)}+H-\phi^v_\sigma \cdot p \right) d\sigma\right)=0}.
\end{equation*}
Then, integrating between 0 and $+\infty$ gives
\begin{equation}\label{vecteurpropre}
Q(v)=\int_0^{+\infty}\frac{M(\phi^v_t)}{\tM(\phi^v_t)}\mathrm{exp}\left( -\int_0^t \left( \frac{M(\phi^v_s)}{\tM(\phi^v_s)}+H-\phi^v_s \cdot p \right) ds\right)dt\int_V \tM'Q'd\nu(v').
\end{equation}
Integrating \Cref{vecteurpropre} against $\tM$ finally gives
\begin{equation*}
1=\int_V\tM(v)\int_0^{+\infty}\frac{M(\phi^v_t)}{\tM(\phi^v_t)}\mathrm{exp}\left( -\int_0^t \left( \frac{M(\phi^v_s)}{\tM(\phi^v_s)}+H-\phi^v_s \cdot p \right) ds\right)dtd\nu(v).
\end{equation*}
In other terms, solving the spectral problem is equivalent to finding $H\in\R$ such that $\int_V \tM'Q'_{p,H}d\nu(v')=1$ holds, where
\begin{equation*}
Q_{p,H}(v)=\int_0^{+\infty}\frac{M(\phi^v_t)}{\tM(\phi^v_t)}\mathrm{exp}\left( -\int_0^t \left( \frac{M(\phi^v_s)}{\tM(\phi^v_s)}+H-\phi^v_s \cdot p \right) ds\right)dt.
\end{equation*}
It is straight-forward to check that $H\mapsto \int_V \tM'Q'_{p,H}d\nu(v')$ is monotically decreasing and continuous. As a result, if there exists $H$ such that $\int_V \tM'Q'_{p,H}=1$, then such $H$ is unique. Let us recall the definition of our hamiltonian:
\begin{equation*}
\mathcal{H}(p):=\mathrm{inf}\left\{ H\in\R,\quad \int\tM(v')Q_{p,H}(v')d\nu(v') \leq 1 \right\}.
\end{equation*}

\begin{proposition}\label{resolution} Resolution of the spectral problem in $C^1(V)$
\begin{enumerate}
\item[(i)] If $p\in \mathrm{Sing}(M,\Gamma)^c$, then $\int_V \tM'Q'_{p,\mathcal{H}(p)}d\nu(v')=1$, \em i.e. \em the couple $(Q_{p,\mathcal{H}(p)},\mathcal{H}(p))$ is a solution to the spectral problem (\ref{probspec}).
\item[(ii)] If $p\in \mathrm{Sing}(M,\Gamma)$, then $\underset{v\in V}{\mathrm{sup}}\,Q_{p,\mathcal{H}(p)}=+\infty$, \em i.e. \em there is no solution of the spectral problem (\ref{probspec}) in $C^1(V)$.

\end{enumerate}
\end{proposition}

\begin{proof}
(i) Let $p\in \mathrm{Sing}(M,\Gamma)^c$. By definition, there exists $H_0\in\R$ such that $\int_V \tM'Q'_{p,H_0}d\nu'>1$. By continuity and monotonicity of $H\mapsto\int_V \tM'Q'_{p,H}d\nu'$, this means that, for all $H_0< H <\mathcal{H}(p)$,
\begin{equation*}
+\infty>\int_V \tM'Q'_{p,H_0}d\nu'>\int_V \tM'Q'_{p,H}d\nu'>1,
\end{equation*}
the last inequality being true by definition since $H<\mathcal{H}(p)$ (recall \eqref{hamiltonian}). Finally,
\begin{equation*}
1\geq \int_V \tM'Q'_{p,\mathcal{H}(p)}d\nu'=\underset{H\searrow \mathcal{H}(p)}{\mathrm{lim}}\int_V \tM'Q'_{p,H}d\nu'\geq 1,
\end{equation*}
which proves (i).\\
(ii) Suppose that $p\in \mathrm{Sing}(M,\Gamma)$ and that $Q_{p,\mathcal{H}(p)}$ is bounded. We let $\delta>0$. Then, $\mathrm{sup}_{v\in V}Q_{p,\mathcal{H}(p)-\delta}=+\infty$. Indeed, in the opposite case $Q_{p,\mathcal{H}(p)-\delta}$ is bounded and hence, integrable on $V$ which is not possible since $p\in \mathrm{Sing}(M,\Gamma)$. Where defined, the function $\mathcal{Z}_\delta:=Q_{p,\mathcal{H}(p)-\delta}-Q_{p,\mathcal{H}(p)}$ satisfies
\begin{equation*}
\left(\frac{M(v)}{\tM(v)}+\mathcal{H}(p)-\delta-v\cdot p\right)\mathcal{Z}_{\delta}+\Gamma\cdot \nabla_v \mathcal{Z}_{\delta}=\delta Q_{p,\mathcal{H}(p)}\leq \delta \left\Vert Q_{p,\mathcal{H}(p)} \right\Vert_{\infty}.
\end{equation*}
By the method of characteristics, this implies that
\begin{eqnarray*}
\mathcal{Z}_{\delta}(v)&\leq &\delta \left\Vert Q_{p,\mathcal{H}(p)} \right\Vert_{\infty}\int_0^{+\infty}\mathrm{exp}\left(-\int_0^t \left(\frac{M(\phi^v_s)}{\tM(\phi^v_s)}+\mathcal{H}(p)-\delta - \phi^v_s\cdot p\right)ds\right)dt\\
&=& \delta \left\Vert Q_{p,\mathcal{H}(p)} \right\Vert_{\infty} \frac{\tM(v)}{M(v)}Q_{p,\mathcal{H}(p)}(v)\\
&\leq& \delta\left\Vert Q_{p,\mathcal{H}(p)} \right\Vert_{\infty}\frac{\mathrm{max}\,\tM}{\mathrm{min}\,M}Q_{p,\mathcal{H}(p)-\delta}(v),
\end{eqnarray*}
for all $v$ where $Q_{p,\mathcal{H}(p)-\delta}(v)<+\infty$. Hence,
\begin{equation*}
Q_{p,\mathcal{H}(p)}(v)\geq \left(1-\delta\left\Vert Q_{p,\mathcal{H}(p)} \right\Vert_{\infty}\frac{\mathrm{max}\,\tM}{\mathrm{min}\,M}\right) Q_{p,\mathcal{H}(p)-\delta}(v).
\end{equation*}
Since $Q_{\mathcal{H}(p)}$ is bounded and $\mathrm{sup}_{v\in V}Q_{p,\mathcal{H}(p)-\delta}=+\infty$, this is absurd.

\end{proof}

We now discuss the resolution of the spectral problem in the set of positive measures.

\begin{lemma}\label{utile}
Let $p\in \mathrm{Sing}(M,\Gamma)^c$ and $v\in V\setminus\mathcal{D}(Q_{p,\mathcal{H}(p)})$, \em i.e. \em $Q_{p,\mathcal{H}(p)}(v)=+\infty$. Thanks to the Poincar\'e-Bendixson condition (\ref{Poincare}), either $\phi^v_t$ converges to some $v_0\in V$ or the limit set of $(\phi^v_t)_t$ is the periodic orbit of some $v_0\in V$. Either way, $Q_{p,\mathcal{H}(p)}(v_0)=+\infty$.
\end{lemma}
\begin{proof}
The result holds since
\begin{equation*}
\underset{t\to +\infty}{\mathrm{lim}}\,\frac{1}{t}\int_0^t \left(\frac{M(\phi^v_s)}{\tM(\phi^v_s)}+\mathcal{H}(p)-\phi^v_s \cdot p\right)ds=\underset{t\to +\infty}{\mathrm{lim}}\,\frac{1}{t}\int_0^t \left(\frac{M(\phi^{v_0}_s)}{\tM(\phi^{v_0}_s)}+\mathcal{H}(p)-\phi^{v_0}_s \cdot p\right)ds,
\end{equation*}
which implies that
\begin{equation*}
\mathrm{exp}\left(-\int_0^t \left(\frac{M(\phi^v_s)}{\tM(\phi^v_s)}+\mathcal{H}(p)-\phi^v_s \cdot p\right)ds\right)\underset{t\to+\infty}{\sim}\mathrm{exp}\left(-\int_0^t \left(\frac{M(\phi^{v_0}_s)}{\tM(\phi^{v_0}_s)}+\mathcal{H}(p)-\phi^{v_0}_s \cdot p\right)ds\right).
\end{equation*}
\end{proof}

\begin{lemma}\label{utile2}
Let now $v_0$ be defined as in \Cref{utile} and let $w(v_0)\in V$ be the vector defined after the mixing property (\ref{mixing}). Then, $\frac{M(w)}{\tM(w)}+\mathcal{H}(p)-w\cdot p=0$.
\end{lemma}
\begin{proof}
Let us assume that $\frac{M(w)}{\tM(w)}+\mathcal{H}(p)-v\cdot p=\delta>0$. Then,
\begin{equation*}
\mathrm{exp}\left(-\int_0^t \left(\frac{M(\phi^{v_0}_s)}{\tM(\phi^{v_0}_s)}+\mathcal{H}(p)-\phi^{v_0}_s \cdot p\right)ds\right)\underset{t\to +\infty}{\sim}e^{-\delta t},
\end{equation*}
hence $Q_{p,\mathcal{H}(p)}(v_0)=\int_{[0,+\infty)}\frac{M(\phi^{v_0}_t)}{\tM(\phi^{v_0}_t)}\mathrm{exp}\left(-\int_{[0,t]}\left(\frac{M(\phi^{v_0}_s)}{\tM(\phi^{v_0}_s)}+\mathcal{H}(p)-\phi^{v_0}_s \cdot p\right)ds\right)dt<+\infty$, which is absurd after \Cref{utile}.
\end{proof}

\begin{proposition}\label{inthesenseofd}
Resolution of the problem in the set of positive measures.

Let $p\in \mathrm{Sing}(M,\Gamma)$. Let $v_0$ be defined after \Cref{utile}.
\begin{enumerate}
\item[(i)]  If $v_0$ is such that $\Gamma(v_0)=0$ and $\frac{M(v_0)}{\tM(v_0)}+\mathcal{H}(p)-v_0\cdot p=0$, then the measure $\mu:=\delta_{v_0}$, where $\delta_{v_0}$ is the dirac mass at $v_0$ satisfies
\begin{equation}\label{distribution}
\left(\frac{M}{\tM}+\mathcal{H}(p)-v\cdot p\right)Q + \Gamma\cdot \nabla_v Q =0
\end{equation}
in the sense of distributions.
\item[(ii)] If $v_0 \in V$ belongs to a periodic orbit of period $T$ and $\phi^{v_0}_t\in V\setminus \mathcal{D}(Q_{p,\mathcal{H}(p)})$ for all $0\leq t \leq T$, then the uniform probability measure $\mu$ on the set $\left\{ \phi^{v_0}_t,\; 0\leq t \leq T \right\}$ satisfies (\ref{distribution}) in the sense of distributions.
\item[(iii)] Either way, the positive measure $\tilde{\mu}:=Q_{p,\mathcal{H}(p)}d\nu+\left(1-\int_V Q'_{p,\mathcal{H}(p)}d\nu'\right)\mu$ associated with the eigenvalue $\mathcal{H}(p)$ is a solution to the spectral problem (\ref{probspec}).
\end{enumerate}
\end{proposition}

\begin{proof}
(i) This is trivial since
\begin{equation*}
\left(\frac{M(v_0)}{\tM(v_0)}+\mathcal{H}(p)-\phi^{v_0}_t\cdot p\right)\psi(v_0)+\Gamma(v_0)\cdot \na_v \psi(v_0)=0\times \psi(v_0) + 0\cdot \na_v \psi(v_0)=0,
\end{equation*}
for all $\psi \in C^{\infty}(V)$.

(ii) Let $\psi \in C^{\infty}(V)$. Then,
\begin{eqnarray*}
\int_{\left\{ \phi^{v_0}_t,\; 0\leq t \leq T \right\}}\left(\frac{M'}{\tM'}-\mathcal{H}(p)-v'\cdot p\right)\psi'd\mu' +\int_{\left\{ \phi^{v_0}_t,\; 0\leq t \leq T \right\}}\Gamma'\cdot \nabla_v \psi'd\mu'\\
=\int_0^T \left(\frac{M(\phi^{v_0}_t)}{\tM(\phi^{v_0}_t)}+\mathcal{H}(p)-\phi^{v_0}_t\cdot p\right)\psi(\phi^{v_0}_t)dt+\int_0^T \Gamma(\phi^{v_0}_t)\cdot \nabla_v \psi(\phi^{v_0}_t)dt.
\end{eqnarray*}
Now, 
\begin{equation*}
\int_0^T \Gamma(\phi^{v_0}_t)\cdot \nabla_v \psi(\phi^{v_0}_t)dt=-\int_0^T \overset{\cdot}{\phi^{v_0}_t}\cdot \nabla_v \psi(\phi^{v_0}_t)dt=\left[\psi(\phi^{v_0}_t) \right]_{t=0}^{t=T}=0,
\end{equation*}
since $\phi^{v_0}$ is periodic with period $T$. Moreover, since
\begin{equation*}
\frac{1}{nT}\int_0^{nT}\left(\frac{M(\phi^{v_0}_t)}{\tM(\phi^{v_0}_t)}+\mathcal{H}(p)-\phi^{v_0}_t\right)dt=\frac{n}{nT}\int_{0}^T\left(\frac{M(\phi^{v_0}_t)}{\tM(\phi^{v_0}_t)}+\mathcal{H}(p)-\phi^{v_0}_t\right)dt,
\end{equation*}
for all $n\in \mathbb{N}$, and
\begin{equation*}
\frac{1}{nT}\int_0^{nT}\left(\frac{M(\phi^{v_0}_t)}{\tM(\phi^{v_0}_t)}+\mathcal{H}(p)-\phi^{v_0}_t\right)dt\underset{n\to +\infty}{\longrightarrow}\frac{M(w)}{\tM(w)}+\mathcal{H}(p)-w\cdot p,
\end{equation*}
we finally get
\begin{equation*}
\int_0^T \left(\frac{M(\phi^{v_0}_t)}{\tM(\phi^{v_0}_t)}+\mathcal{H}(p)-\phi^{v_0}_t\cdot p\right)\psi(\phi^{v_0}_t)dt=T\left(\frac{M(w)}{\tM(w)}+\mathcal{H}(p)-w\cdot p\right)\psi(w)=0,
\end{equation*}
which ends the proof.

(iii) This is straight-forward since $\mu$ solves \Cref{distribution} and since
\begin{equation*}
\left(\frac{M}{\tM}+\mathcal{H}(p)-v\cdot p\right)Q_{p,\mathcal{H}(p)} +\Gamma\cdot \na_v Q_{p,\mathcal{H}(p)} = \frac{M}{\tM},
\end{equation*}
for all $v\in \mathcal{D}(Q_{p,\mathcal{H}(p)})$.
\end{proof}

\begin{rem}
We do not necessarily have uniqueness of the spectral problem (\ref{probspec}) in the set of positive measures as there might exist several points in $V\setminus \mathcal{D}(Q_{p,\mathcal{H}(p))}$, which do not belong to the same orbit. However, we will only use the solutions of \Cref{resolution} and \Cref{inthesenseofd} in the rest of the paper. We will use the perturbed test function method of Evans \cite{evans_perturbed_1989} to build a sub- and a super-solution of the Hamilton-Jacobi equation \eqref{HJ}. When $p\in \mathrm{Sing}(M,\Gamma)^c$, we will use the $C^1$ solution of \Cref{resolution} to build the perturbed test function in question. It is worth mentioning that when $p\in \mathrm{Sing}(M,\Gamma)$, we will use the solution in the set of positive measures given by \Cref{inthesenseofd}. However, we will only use the regular part $Q_{p,\mathcal{H}(p)}$ in the super-solution procedure, whereas we will only use the singular part $\mu$ in the sub-solution procedure.
\end{rem}

\subsection{Examples}

Such a hamiltonian was already studied in a less general setting. Here are two examples taken from \cite{bouin_kinetic_2012,caillerie_large_2017-1} and \cite{caillerie_stochastic_2017}.

\begin{example} \textbf{The special case $\Gamma\equiv 0$.} \\
Suppose that $V$ is a compact set such that $0$ belongs to the interior of the convex hull of $V$. In this case, it is straight-forward to check that $M$ is a solution to \eqref{tm} hence $\tM=M$. Moreover, $\mathrm{Sing}(M,\Gamma)=\left\{ p\in \R^d,\; \int_V \frac{M(v')}{\mu(p)-v\cdot p}\leq 1 \right\}$, where $\mu(p):= \mathrm{max}_{v\in V}\left\{v\cdot p\right\}$. The hamiltonian is then defined by:
\begin{equation}\label{bouincalvez}
\int_V \frac{M(v)}{1+\mathcal{H}(p)-v\cdot p}d\nu(v)=1,\quad \mbox{if }p\notin \mathrm{Sing}(M,\Gamma),
\end{equation}
\begin{equation}\label{caillerie}
\mathcal{H}(p)=\mu(p)-1,\quad \mbox{if }p\in \mathrm{Sing}(M,\Gamma).
\end{equation}
When $V=[-1,1]$ and $M\equiv \frac{1}{2}$, then $\mathrm{Sing}(M,\Gamma)=\emptyset$ and $\mathcal{H}(p)=\frac{p-\mathrm{tanh}p}{\mathrm{tanh}p}$ for all $p\in \mathbb{R}^d$.
\end{example}
One can refer to \cite{bouin_kinetic_2012,caillerie_large_2017-1} for more details. Let us emphasize that the hamiltonian \eqref{bouincalvez}-\eqref{caillerie} is consistent with ours. Indeed, when $\Gamma\equiv 0$, then $\tM\equiv M$ and $\phi^v_s=v$, for all $s$, hence
$\displaystyle{\int_V \tM(v)\int_0^{+\infty}\frac{M(\phi^v_t)}{\tM(\phi^v_t)}\mathrm{exp}\left(-\int_0^t \left(\frac{M(\phi^v_s)}{\tM(\phi^v_s)}+\mathcal{H}(p)-\phi^v_s\cdot p\right)ds \right)dtd\nu(v)}$
\begin{equation*}
= \int_V M(v)\int_0^{+\infty}\mathrm{exp}\left(-\int_0^t \left(\frac{M(v)}{M(v)}+\mathcal{H}(p)-v\cdot p\right)ds \right)dtd\nu(v)
\end{equation*}
\begin{flushright}
$\displaystyle{=\int_V \frac{M(v)}{1+\mathcal{H}(p)-v\cdot p}d\nu(v)}.$\\
\par\end{flushright}
It is also straightforward to check that the hamiltonian from \cite{bouin_kinetic_2012,caillerie_large_2017-1} and ours coincide on $\mathrm{Sing}(M,\Gamma)$.

\begin{example}\label{chap3}
Let $d=3$, $V$ be the unit sphere that we parametrize with the usual spherical coordinates: $V=\left\{ (\theta,\varphi)\in [0,2\pi]\times[0,\pi] \right\}$, $v(\theta,\varphi):=(\mathrm{sin}(\varphi)\mathrm{cos}(\theta),\mathrm{sin}(\varphi)\mathrm{sin}(\theta),\mathrm{cos}\varphi)$ and let $M$ depend only on $\varphi$ and $\Gamma(\theta,\varphi)=(\mathrm{sin}\varphi,0)$. If $p\notin \mathrm{Sing}(M,\Gamma)$, then $\mathcal{H}(p)$ is implicitly defined by
\begin{equation*}
\int_V M(\varphi)\int_0^{+\infty}\mathrm{exp}\left(-\int_0^t \left(1+\mathcal{H}(p)-v(\theta-s,\varphi)\cdot p\right)ds\right)dt\frac{\mathrm{sin}(\varphi)d\theta d\varphi}{4\pi}=1,
\end{equation*}
and if $p\in \mathrm{Sing}(M,\G)$, then $\mathcal{H}(p)=\left\vert p\cdot e_3 \right\vert-1$, where $e_3=(0,0,1)$.
\end{example}
One can find a proof of this result in \cite{caillerie_stochastic_2017}, Chapter 3. Let us emphasize that the addition of the force term is a singular perturbation in our Hamilton-Jacobi framework since the hamiltonian of \Cref{chap3} is different, at least on $\mathrm{Sing}(M,\G)$, from the one obtained when $\G\equiv 0$.

\subsection{Properties of the hamiltonian}

\begin{proposition}\label{lipschitz}
The hamiltonian has the following properties:
\begin{enumerate}
\item[(i)] $0\in \mathrm{Sing}(M,\Gamma)^c$ and $\mathcal{H}(0)=0$.
\item[(ii)] $\mathcal{H}$ is continuous on $\R^d$ and $C^1$ on $\R^d\setminus\mathrm{Sing}(M,\G)$.
\item[(iii)] $\mathcal{H}$ is lipschitz-continuous.
\end{enumerate}
\end{proposition}

\begin{proof}
(i) This result is trivial once one notices that
\begin{equation*}
\int_V \tM(v) \int_0^{+\infty}\frac{M(\phi^v_t)}{\tM(\phi^v_t)}\mathrm{exp}\left(-\int_0^t \frac{M(\phi^v_s)}{\tM(\phi^v_s)}ds\right)dtd\nu(v)=\int_V \tM(v)\frac{M(v)}{\tM(v)}d\nu(v)=1.
\end{equation*}

(ii) On $\mathrm{Sing}(M,\G)^c$, the function $\mathcal{H}$ is implicitly defined by the relation
\begin{equation}\label{implicit1}
\int_V \tM(v)Q_{p,\mathcal{H}(p)}(v)d\nu(v)=1.
\end{equation}
On $\overset{\circ}{\mathrm{Sing}(M,\G)}$, $\mathcal{H}(p)$ is implicitly defined by the relation
\begin{equation}\label{implicit2}
\int_V \tM(v)Q_{p,\mathcal{H}(p)}(v)d\nu(v)-\underset{H\in \mathfrak{B}(p)}{\mathrm{max}}\left\{ \int_V \tM(v)Q_{p,H}(v)d\nu(v) \right\}=0,
\end{equation}
where $\mathfrak{B}(p)=\left\{ H\in \mathbb{R},\quad\int_V \tM(v)Q_{p,H}(v)d\nu(v)<+\infty\right\}$.
Hence, by the implicit function Theorem, $\mathcal{H}$ is $C^1$ on $\overset{\circ}{\mathrm{Sing}(M,\G)}\cup \mathrm{Sing}(p)^c=\R^d\setminus \partial\mathrm{Sing}(M,\G)$. Moreover, since for all $p\in \partial\mathrm{Sing}(M,\G)$,
\begin{equation*}
\underset{H\in\mathfrak{B}(p)}{\mathrm{max}}\left\{ \int_V \tM(v)Q_{p,H}(v)d\nu(v) \right\}=1,
\end{equation*}
we also conclude that $\mathcal{H}$ is continuous on $\R^d$.

(iii) Differentiating \eqref{implicit1} and \eqref{implicit2} with respect to $p$ and recalling \eqref{defQ}, we get for all $p\in \R^d\setminus \partial \mathrm{Sing}(M,\Gamma)$,
\begin{equation*}
\int_V \tM(v)\int_0^{+\infty} \frac{M(\phi^v_t)}{\tM(\phi^v_t)} \left(\int_0^t \left(\nabla\mathcal{H}(p)-\phi^v_s\right)ds\right)\mathrm{exp}\left( \int_0^t \left(\frac{M(\phi^v_s)}{\tM(\phi^v_s)}+\mathcal{H}(p)-\phi^v_s\cdot p\right)ds \right)dtd\nu(v)=0.
\end{equation*}
Hence, $\left\Vert \nabla \mathcal{H} \right\Vert_{\infty}\leq \underset{v\in V}{\mathrm{sup}}\;\left\vert v \right\vert$ on $\R^d\setminus \partial \mathrm{Sing}(M,\Gamma)$, from which we deduce that $\mathcal{H}$ is lipschitz-continuous.

\end{proof}

\section{Convergence to the Hamilton-Jacobi limit}

\subsection{A priori estimates}\label{apriori}

\begin{proposition}\label{aprioriprop}
Let us assume that (\ref{infM}) and (\ref{divgamma}) hold and that $\tM$ satisfies (\ref{tm}). Let $\varphi^{\eps}$ satisfy \Cref{main}. Let us assume that the initial condition is well-prepared: $\varphi^{\eps}(0,x,v)=\varphi_0(x)\geq 0$. Then, $\varphi^{\eps}$ is uniformly bounded with respect to $x$, $v$, and $\eps$. More precisely, for all $0\leq t \leq T$,
\begin{equation}
0\leq \varphi^{\eps}(t,\cdot,\cdot) \leq \left\Vert \varphi_0 \right\Vert_{\infty} + \frac{\mathrm{max}\,M}{\mathrm{min}\,\tM}T
\end{equation}
\end{proposition}

\begin{proof}
Let us recall that, under assumptions (\ref{infM}) and (\ref{divgamma}), the results of \Cref{borneinf} and \Cref{Lemmma} hold.

Let $\left(\mathcal{X}^{\eps},\mathcal{V}^\eps\right)$ be the characteristics associated with (\ref{main}):
\begin{equation*}
\begin{cases}
\overset{\cdot}{\mathcal{X}_{s,t}^{x,v}}=\mathcal{V}_{s,t}^{x,v},\\
\mathcal{X}^{x,v}_{t,t}=x,\\
\overset{\cdot}{\mathcal{V}_{s,t}^{x,v}}=\frac{\Gamma\left(\mathcal{V}_{s,t}^{x,v}\right)}{\eps},\\
\mathcal{V}_{t,t}^{x,v}=v.
\end{cases}
\end{equation*}
Here, we dropped the $\eps$ for readability reasons. Using the method of characteristics, we get the following relation
\begin{eqnarray*}
\varphi^{\eps}(t,x,v)&=&\varphi_0(\mathcal{X}^{x,v}_{0,t})\\
&+&\int_0^t\frac{M\left(\mathcal{V}_{s,t}^{x,v}\right)}{\tM\left(\mathcal{V}_{s,t}^{x,v}\right)}\int_V\tM(v')\left(1-\mathrm{exp}\left(\frac{\varphi^\eps\left(s,\mathcal{X}^{x,v}_{s,t},\mathcal{V}^{x,v}_{s,t}\right)-\varphi^{\eps}\left(s,\mathcal{X}^{x,v}_{s,t},v'\right)}{\eps}\right)\right)d\nu(v')ds\\
& \leq & \varphi_0(\mathcal{X}^{x,v}_{0,t})+ \int_0^t\frac{M\left(\mathcal{V}_{s,t}^{x,v}\right)}{\tM\left(\mathcal{V}_{s,t}^{x,v}\right)}\int_V\tM(v')d\nu(v')ds\\
& \leq &  \left\Vert \varphi_0\right\Vert_{\infty} + \int_0^T \frac{\mathrm{max}\,M}{\mathrm{min}\,\tM}\int_V \tM(v')d\nu(v')ds= \left\Vert \varphi_0 \right\Vert_{\infty} + \frac{\mathrm{max}\,M}{\mathrm{min}\,\tM}T,
\end{eqnarray*}
so we have an upper bound on $\varphi^\eps$.

We get the lower bound by noticing that $0$ trivially satisfies \eqref{main}.

 

\end{proof}

\subsection{Proof of \Cref{thm}}

In this Section, we prove \Cref{thm} using the half-relaxed limits method of Barles and Perthame \cite{barles_exit_1988} in the same spirit as in \cite{bouin_spreading_2017}. Additionally, we use the method of the perturbed test function of Evans \cite{evans_perturbed_1989} using the same ideas as in \cite{bouin_kinetic_2012,bouin_spreading_2017,caillerie_large_2017-1}.

Thanks to \Cref{aprioriprop}, the sequence $(\varphi^\eps)_{\eps}$ is uniformly bounded in $L^\infty$ with respect to $\eps$. We can thus define its lower and upper semi continuous envelopes:
\begin{equation}\label{envelopes}
\varphi^*(t,x,v) = \limsup_{\substack{\eps \to 0\\(s,y,w) \to (t,x,v)}} \varphi^\eps(s,y,w), \qquad \varphi_*(t,x,v) = \liminf_{\substack{\eps \to 0\\(s,y,w) \to (t,x,v)}} \varphi^\eps(s,y,w).
\end{equation}
We will prove that $\varphi^*$ and $\varphi_*$ are respectively a sub- and a super-solution of the Hamilton-Jacobi equation. In order to do that, we need to prove that neither functions depend on the velocity variable. For this, we will use a similar proof to \cite{bouin_spreading_2017}. We write it here for the sake of self-containedness.

\begin{lemma}\label{indep}
Both $\varphi^*$ and $\varphi_*$ are constant with respect to the velocity variable on $\R^*_+\times\R^d$.
\end{lemma}

\begin{proof}
Let $(t^0,x^0,v^0)\in\R^*_+\times\R^d\times V$ and $\psi\in C^1\left(\R^*_+\times\R^d\times V\right)$ be a test function such that $\varphi^*-\psi$ has a strict local maximum at $(t^0,x^0,v^0)$. Then, there exists a sequence $(t^\eps,x^\eps,v^\eps)$ such that $\varphi^\eps-\psi$ attains its maximum at $(t^\eps,x^\eps,v^\eps)$ and such that $(t^\eps,x^\eps,v^\eps)\to(t^0,x^0,v^0)$. Thus, $\mathrm{lim}_{\eps\to0}\varphi^{\eps}(t^\eps,x^\eps,v^\eps)=\varphi^*(t,x,v)$. Moreover, at point $(t^\eps,x^\eps,v^\eps)$, we have:
\begin{equation*}
\partial_t \psi + v^{\eps} \cdot \nabla_x \psi +\frac{\G(v^\eps)}{\eps}\cdot \nabla_v \psi = \frac{M(v^\eps)}{\tM(v^\eps)}\int_V \tM(v')\left(1- e^{\frac{\varphi^{\eps} -\varphi^{\eps'}}{\eps}} \right)d\nu(v').
\end{equation*}
From this, and using the fact that
\begin{equation*}
0<\mathrm{min}\,M\leq M\leq \mathrm{max}\,M<+\infty,
\end{equation*}
\begin{equation*}
0<\mathrm{min}\,\tM\leq \tM\leq \mathrm{max}\,\tM<+\infty,
\end{equation*}
we deduce that $\eps\int_{V'} \tM(v') e^{\frac{\varphi^{\eps}(t^\eps,x^\eps,v^\eps) -\varphi^{\eps}(t^\eps,x^\eps,v')}{\eps}} d\nu(v')$ is uniformly bounded for all $V'\subset V$. By the Jensen inequality,
\begin{equation*}
\eps \mathrm{exp}\left(\frac{1}{\eps\left\vert V' \right\vert_{\tM}}\int_{V'}\tM(v')\left(\varphi^{\eps}(t^\eps,x^\eps,v^\eps)-\varphi^{\eps}(t^\eps,x^\eps,v')\right)d\nu(v')\right)\leq \frac{\eps}{\left\vert V' \right\vert_{\tM}} \int_{V'}\tM(v')e^{\frac{\varphi^{\eps}(t^\eps,x^\eps,v^\eps) -\varphi^{\eps}(t^\eps,x^\eps,v')}{\eps}} d\nu(v'),
\end{equation*}
where $\left\vert V' \right\vert_{\tM}:=\displaystyle{\int_{V'}\tM(v')d\nu(v')}$. We deduce that
\begin{equation*}
\underset{\eps\to0}{\mathrm{lim}\,\mathrm{sup}} \int_{V'}\tM(v')\left(\varphi^{\eps}(t^\eps,x^\eps,v^\eps)-\varphi^{\eps}(t^\eps,x^\eps,v')\right)d\nu(v') \leq0
\end{equation*}
We write
\begin{align*}
\int_{V'} \tM(v')\left(\varphi^{\eps}(v^{\eps}) -\varphi^{\eps}(v') \right) d\nu(v') &= \int_{V'} \tM(v')\left[\left( \varphi^{\eps}(v^{\eps}) - \psi(v^{\eps})\right) - \left( \varphi^{\eps}(v') - \psi(v') \right) + \left( \psi(v^{\eps}) - \psi(v')\right)\right] d\nu(v')\\
&= \int_{V'}  \tM(v')\left[\left( \varphi^{\eps}(v^{\eps}) - \psi(v^{\eps})\right) - \left( \varphi^{\eps}(v') - \psi(v') \right) \right] d\nu(v')\\ &\hfill + \int_{V'} \tM(v') \left( \psi(v^{\eps}) - \psi(v')\right)  d\nu(v')
\end{align*}
We can thus use the Fatou Lemma, together with $- \limsup_{\eps \to 0} \varphi^{\eps}(t^{\eps},x^{\eps},v') \geq - \varphi^*(t^0,x^0,v')$ to get
\begin{align*}
\left(\int_{V'} \tM(v')d\nu(v')\right) \varphi^*(v^0) - \int_{V'}\tM(v')\varphi^*(v') d\nu(v') & = \int_{V'} \tM(v')\left(\varphi^*(v^0) -\varphi^*(v') \right)  d\nu(v') \\
&\leq
\int_{V'} \tM(v') \liminf_{\eps \to 0}  \left(\varphi^{\eps}(v^{\eps}) -\varphi^{\eps}(v') \right) d\nu(v') \\
&\leq \liminf_{\eps \to 0} \left( \int_{V'} \tM(v') \left(\varphi^{\eps}(v^{\eps}) -\varphi^{\eps}(v') \right) d\nu(v') \right)\\ 
&\leq \limsup_{\eps \to 0} \left( \int_{V'} \tM(v') \left(\varphi^{\eps}(v^{\eps}) -\varphi^{\eps}(v') \right) d\nu(v') \right)\\
&\leq 0, 
\end{align*}
We shall deduce, since the latter is true for any $\vert V' \vert$ that
\begin{equation*}
\varphi^*(t^0,x^0,v^0) \leq \inf_{V} \varphi^*(t^0,x^0,\cdot)
\end{equation*}
and thus $\varphi^*$ is constant in velocity.

To prove that $\varphi_*$ is constant with respect to the velocity variable, we use the same technique with a test function $\psi$ such that $\varphi^\eps-\psi$ has a local strict minimum at $(t^0,x^0,v^0)$.
\end{proof}

We shall now prove the following fact
\begin{prop}\label{viscosity}
Let $\varphi^\eps$ be a solution of (\ref{main}) and let $\vp^*$ and $\vp_*$ be defined by (\ref{envelopes}).
\begin{enumerate}\label{prop:semilimits}
\item[(i)] The function $\varphi_*$ is a viscosity super-solution to the Hamilton-Jacobi equation \eqref{HJ} on $\R_+^* \times \R^n$.
\item[(ii)] The function $\varphi^*$ is a viscosity sub-solution to the Hamilton-Jacobi equation \eqref{HJ} on $\R_+^* \times \R^n$.
\end{enumerate}
\end{prop}

\begin{proof}

(i) Let $\psi$ be a test function such that $\vp_*-\psi$ has a local minimum at point $(t^0,x^0)\in\R_+^*\times\R^d$. We set $p^0:=\na_x \psi(t^0,x^0)$. For all $H\geq\mathcal{H}(p^0)$, let us define $\psi^{\eps}_H:=\psi+\eps\eta_H$, where $\eta_H:=-\mathrm{log}(Q_{p^0,H})$ and
\begin{equation}\label{QQ}
Q_{p^0,H}(v):=\int_0^{+\infty}\frac{M(\phi^v_t)}{\tM(\phi^v_t)}\mathrm{exp}\left(-\int_0^t \left(\frac{M(\phi^v_s)}{\tM(\phi^v_s)}+H-\phi^v_s\cdot p^0\right)ds\right)dt, \quad\forall v\in V.
\end{equation}
For all $H>\mathcal{H}(p^0)$, by construction of $\eta_H$, we have
\begin{equation*}
\int_V \tM'e^{-\eta'_H}d\nu(v')= \int_V \tM'Q'_{p^0,H}d\nu(v')<\int_V \tM'Q'_{p^0,\mathcal{H}(p^0)}d\nu(v')= 1,\quad\mbox{if }p^0\notin\mathrm{Sing}(M,\Gamma),
\end{equation*}
or
\begin{equation*}
\int_V \tM'e^{-\eta'_H}d\nu(v')= \int_V \tM'Q'_{p^0,H}d\nu(v')<\int_V \tM'Q'_{p^0,\mathcal{H}(p^0)}d\nu(v')\leq 1,\quad\mbox{if }p^0\in\mathrm{Sing}(M,\Gamma).
\end{equation*}
Moreover, $Q_{p^0,H}\in C^1(V)$ and
\begin{equation}\label{QQQ}
Q_{p^0,H}\left(\frac{M(v)}{\tM(v)}+H-v\cdot p^0\right)+\G\cdot \na_v Q_{p^0,H}=\frac{M(v)}{\tM(v)},\quad \forall v\in V.
\end{equation}

By uniform convergence of $\psi_H^\eps$ toward $\psi$ and by the definition of $\varphi_*$, the function $\vp^\eps-\psi^\eps_H$ has a local minimum located at a point $(t^\eps,x^\eps,v^\eps)\in\R_+^*\times\R^d\times V$, satisfying $t^\eps\to t^0$ and $x^\eps \to x^0$. The extremal property of $(t^\eps,x^\eps,v^\eps)$ implies that
\begin{equation*}
\partial_t\varphi^{\eps}(t^\eps,x^\eps,v^\eps)=\partial_t\psi^{\eps}_H(t^\eps,x^\eps,v^\eps),\quad \nabla_x\varphi^{\eps}(t^\eps,x^\eps,v^\eps)=\nabla_x\psi^{\eps}_H(t^\eps,x^\eps,v^\eps).
\end{equation*}
Moreover, we have
\begin{equation*}
\Gamma(v^\eps)\cdot\nabla_v\varphi^{\eps}(t^\eps,x^\eps,v^\eps)=\Gamma(v^\eps)\cdot\nabla_v\psi^{\eps}_H(t^\eps,x^\eps,v^\eps).
\end{equation*}
Indeed, if $v^\eps\in \overset{\circ}{V}$ or $\Gamma(v^\eps)=0$, then the result is trivial. If $v^\eps\in \partial V$ and $\Gamma(v^\eps)\neq 0$, since $\Gamma(v)\cdot d\overrightarrow{S}(v)=0$ for all $v\in \partial V$, there exists $v_0\in V$, $v_1 \in V$ and $\delta>0$ such that
\begin{equation*}
\begin{cases}
\phi^{v^\eps}_s\in V,& \forall s\in [-\delta,\delta],\\
\phi^{v^\eps}_{-\delta} = v_0,\\
\phi^{v^\eps}_{\delta}=v_1.
\end{cases}
\end{equation*}
The extremal property of $(t^\eps,x^\eps,v^\eps)$ now implies that 
\begin{equation*}
\Gamma(v^\eps)\cdot \nabla_v(\vp^\eps-\psi^\eps_H)(t^\eps,x^\eps,v^\eps)=-\left. \frac{d}{ds}\left(\vp^\eps-\psi^\eps_H\right)(t^\eps,x^\eps,\phi^{v^\eps}_s)\right\vert_{s=0}=0.
\end{equation*}

Finally, since $V$ is a compact set, we know that there exists $v^*\in V$ and a subsequence of $(v^\eps)_\eps$, which we will not relabel, such that $v^\eps\to v^*$.

At point $(t^\eps,x^\eps,v^\eps)$, we have:
\begin{eqnarray}
\partial_t \psi + v^\eps \cdot \nabla_x \psi +\Gamma(v^\eps)\cdot \na_v \eta_H & = & \partial_t \psi^\eps_H + v^\eps\cdot \na_x \psi^\eps_H +\frac{\Gamma(v^\eps)}{\eps}\cdot \na_v \psi^\eps_H \nonumber \\
&\geq&\partial_t \vp^\eps + v^\eps \cdot \na_x \vp^\eps +\frac{\G(v^\eps)}{\eps}\cdot \nabla_v \vp^\eps \label{inneqa}\\
&=&\frac{M(v^\eps)}{\tM(v^\eps)}\left(1-\int_V \tM'e^{\frac{\vp^\eps-\vp^{\eps'}}{\eps}}d\nu(v')\right).\nonumber
\end{eqnarray}
By the minimal property of $(t^\eps,x^\eps,v^\eps)$, we can estimate the right-hand side of the last equation, such that
\begin{eqnarray}
\partial_t \psi+v^\eps\cdot\na_x\psi+\G(v^\eps)\cdot\na_v \eta_H & \geq& \frac{M(v^\eps)}{\tM(v^\eps)}\left(1-\int_V \tM(v')e^{\eta_H(v^\eps)-\eta_H(v')}d\nu(v')\right)\label{inneqb}\\
&\geq & \frac{M(v^\eps)}{\tM(v^\eps)}\left(1-e^{\eta_H(v^\eps)}\right)\label{inneqc}\\
&=&\frac{M(v^\eps)}{\tM(v^\eps)}\left(1-\frac{1}{Q_{p^0,H}(v^\eps)}\right),\nonumber
\end{eqnarray}
so we have at point $(t^\eps,x^\eps,v^\eps)$,
\begin{equation*}
Q_{p^0,H}(v^\eps)\left(\frac{M(v^\eps)}{\tM(v^\eps)}-\partial_t \psi-v^\eps\cdot \nabla_x\psi\right)+\Gamma(v^\eps)\cdot\nabla_v Q_{p^0,H}(v^\eps)\leq \frac{M(v^\eps)}{\tM(v^\eps)}.
\end{equation*}
Taking the limit $\eps\to 0$, we get at point $(t^0,x^0,v^*)$,
\begin{equation}\label{finQQ}
Q_{p^0,H}(v^*)\left(\frac{M(v^*)}{\tM(v^*)}-\partial_t \psi-v^*\cdot p^0\right)+\Gamma(v^*)\cdot\nabla_v Q_{p^0,H}(v^*)\leq \frac{M(v^*)}{\tM(v^*)}.
\end{equation}
 
Combining \eqref{QQQ} and \eqref{finQQ} at $v=v^*$, we get
\begin{equation*}
\partial_t \psi(t^0,x^0)+H \geq 0.
\end{equation*}
Since this is true for any $H>\mathcal{H}(p^0)$, we finally have
\begin{equation*}
\partial_t \psi(t^0,x^0)+\mathcal{H}(p^0) \geq 0,
\end{equation*}
which proves (i).

\medskip

(ii) Let $\psi$ be a test function such that $\varphi^*-\psi$ has a global strict maximum at a point $(t^0,x^0)\in\R_+^*\times\R^d$. We still denote $p^0=\na_x \psi(t^0,x^0)$.

\medskip

{\bf \# First case: $p^0 \notin\mathrm{Sing}\,\left(M,\Gamma\right)$} 

\medskip

Then, from the very definition of $\mathrm{Sing}(M,\Gamma)$ (check \Cref{defiprincipale}), there exists $H_0<\mathcal{H}(p^0)$ such that, for all $H_0<H<\mathcal{H}(p^0)$,
\begin{equation}\label{thanksto}
+\infty>\int_V \tM'Q'_{p^0,H}d\nu(v')>\int_V \tM'Q'_{p^0,\mathcal{H}(p^0)}d\nu(v')=1,
\end{equation}
using the same notation as earlier. We can then conclude using the same arguments as in the proof of (i). We emphasize that the Estimates (\ref{inneqa}) and (\ref{inneqb}) are reverted in this "maximum" case and that (\ref{inneqc}) is reverted thanks to (\ref{thanksto}).

\medskip

{\bf \# Second case: $p^0 \in\mathrm{Sing}\,\left(M,\Gamma\right)$} 

\medskip

%
%

Thanks to \Cref{utile}, there exists $v_0\in V$ such that $Q_{p,\mathcal{H}(p)}(v_0)=+\infty$ and that either $v_0$ is a fixed point of the flow of $-\Gamma$, \em i.e. \em $\Gamma(v_0)=0$, or $v_0$ belongs to a periodic orbit of the flow.

Suppose that $v_0$ is a fixed point, then, after \Cref{utile}, we have
\begin{equation}\label{fatal}
\frac{M(v_0)}{\tM(v_0)}+\mathcal{H}(p)-v_0\cdot p=0.
\end{equation}
Moreover, the function $(t,x)\mapsto \vp^\eps(t,x,v_0) - \psi(t,x)$ has a local maximum at a point $(t^\eps,x^\eps)$ and, by definition of $\vp^*$, we have $t^\eps\to t^0$ and $x^\eps \to x^0$. By the maximal property of $(t^\eps,x^\eps)$, we have at point $(t^\eps,x^\eps,v_0)$,
\begin{eqnarray*}
\partial_t \psi (t^\eps,x^\eps)+v_0 \cdot \na_x \psi(t^\eps,x^\eps) + 0&=&\partial_t \psi (t^\eps,x^\eps)+v_0 \cdot \na_x \psi(t^\eps,x^\eps) + \frac{\Gamma(v_0)}{\eps}\cdot \na_v \varphi^\eps(t^\eps,x^\eps,v_0)\\
&=&\frac{M(v_0)}{\tM(v_0)}\int_V M' \left(1-e^{\frac{\vp^\eps-\vp'^\eps}{\eps}}\right)d\nu'\\
&\leq&\frac{M(v_0)}{\tM(v_0)}.
\end{eqnarray*}
Taking the limit $\eps\to0$ and recalling, (\ref{fatal}), we get
\begin{equation*}
\partial_t \psi(t^0,x^0)+\mathcal{H}\left(\na_x\psi(t^0,x^0)\right)\leq 0,
\end{equation*}
which proves that $\vp^*$ is a viscosity subsolution of (\ref{HJ}).

Suppose now that $v_0$ belongs to a periodic orbit. At point $(t,x,\phi^{v_0}_s)$, we have
\begin{equation*}
\partial_t \vp^\eps(\phi^{v_0}_s)+\phi^{v_0}_s\cdot \na_x \vp^\eps(\phi^{v_0}_s)+\frac{\Gamma(\phi^{v_0}_s)}{\eps}\cdot \na_v \vp^\eps(\phi^{v_0}_s)\leq\frac{M(\phi^{v_0}_s)}{\tM(\phi^{v_0}_s)}
\end{equation*}
Applying $\underset{t\to +\infty}{\mathrm{lim}}\displaystyle{\int_0^t (\cdot) ds}$ to the latter expression gives
\begin{equation*}
\partial_t\vp^\eps(t,x,w)+w\cdot \nabla_x \vp^\eps(t,x,w)\leq \frac{M(w)}{\tM(w)},
\end{equation*}
where $w(v_0)$ is the representative of the orbit defined by the mixing property (\ref{mixing}). Indeed, 
\begin{eqnarray*}
\frac{1}{t}\int_0^t \frac{\Gamma(\phi^{v_0}_s)}{\eps}\cdot \nabla_v \vp^\eps(\phi^{v_0}_s)ds=-\frac{1}{t}\int_0^t \frac{\overset{\cdot}{\phi^{v_0}_s}}{\eps}\cdot \nabla_v \vp^\eps(\phi^{v_0}_s)ds
=-\frac{\vp^\eps(\phi^{v_0}_t)-\vp^\eps(\phi^{v_0}_0)}{t\eps}\underset{t\to+\infty}{\longrightarrow}0.
\end{eqnarray*}
After \Cref{utile2}, we know that $\frac{M(w)}{\tM(w)}+\mathcal{H}(p)-w\cdot p=0$ so we can conclude as in the previous case by considering the function $(t,x)\mapsto \vp^\eps(t,x,w) - \psi(t,x)$.

\end{proof}

\medskip

We can now conclude the proof of \Cref{thm}. 

\begin{proof}[{\bf Proof of \Cref{thm}}]

We refer to Section 4.4.5 in \cite{barles_solutions_1994} and  Theorem B.1 in \cite{evans_pde_1989} for arguments giving strong uniqueness (which means that there exists a comparison principle for sub- and super-solution) of \Cref{HJ} in the viscosity sense. We emphasize that the lipschitz-continuity proven in \Cref{lipschitz} is sufficient for these results. Thanks to \Cref{viscosity}, as $\vp^*$ and $\vp_*$ are respectively a sub- and a super-solution of the Hamilton-Jacobi \Cref{HJ}, the comparison principle yields $\vp^*\leq \vp_*$. However, from their definitions, it is clear that $\vp^*\geq \vp_*$. Hence, the function $\varphi^0:=\varphi^*=\varphi_*$ is the viscosity solution of \Cref{HJ} and $(\vp^\eps)_\eps$ converges uniformly locally as $\eps\to0$ to $\vp^0$, which concludes the proof.

\end{proof}

\section*{Acknowledgement}

The author would like to thank Julien Vovelle and Vincent Calvez for presenting this problem to him and for the many discussions that followed. The author would also like to thank Emeric Bouin and Serge Parmentier, who were very helpful. The author has also received help from Nicolas Champagnat and J\'er\^ome Coville who, when reviewing the author's PhD thesis, gave important advice on the redaction of the proof. This work was performed during the author's PhD in Universit\'e Claude Bernard Lyon 1.

\bibliographystyle{plain}

\bibliography{biblio}

\end{document}